\documentclass[12pt,reqno]{amsart}
\usepackage{amssymb}
\usepackage{amsmath}
\usepackage{amsthm}
\usepackage[left=3cm,top=3cm,right=3cm,bottom=3cm]{geometry}
\usepackage{hyperref}
\usepackage[usenames]{color}
\usepackage{enumitem}
\usepackage[normalem]{ulem} 
\usepackage{graphicx}

\usepackage{etoolbox}
\makeatletter
\patchcmd{\@maketitle}
  {\ifx\@empty\@dedicatory}
  {\ifx\@empty\@date \else {\vskip3ex \centering\footnotesize\@date\par\vskip1ex}\fi
   \ifx\@empty\@dedicatory}
  {}{}
\patchcmd{\@adminfootnotes}
  {\ifx\@empty\@date\else \@footnotetext{\@setdate}\fi}
  {}{}{}
\makeatother

\begin{document}

\newtheorem{theorem}{Theorem}[section]
\newtheorem{lemma}[theorem]{Lemma}
\newtheorem{definition}[theorem]{Definition}
\newtheorem{conjecture}[theorem]{Conjecture}
\newtheorem{proposition}[theorem]{Proposition}
\newtheorem{claim}[theorem]{Claim}
\newtheorem{corollary}[theorem]{Corollary}
\newtheorem{observation}[theorem]{Observation}
\newtheorem{problem}[theorem]{Open Problem}

\newtheorem*{ineq}{\textbf{Inequalities}}

\newtheorem*{c0}{\textbf{Markov Bound}}
\newtheorem*{c1}{\textbf{Chernoff's Bound}}
\newtheorem*{c2}{\textbf{First Moment Method}}
\newtheorem*{c3}{\textbf{Bernstein's Bound}}
\newtheorem*{c4}{\textbf{Second Moment Method}}
\newtheorem*{c5}{\textbf{Method of Moments}}
\newtheorem*{ub}{\textbf{Union Bound}}
\newtheorem*{bc}{\textbf{Binomial Coefficients Approximation}}
\newtheorem*{123conj}{\textbf{1-2-3 Conjecture}}

\theoremstyle{remark}
\newtheorem{rem}[theorem]{Remark}

\newcommand{\weighted}[1]{\mbox{2-WEIGHTED}^{{(#1)}}}
\newcommand{\Whp}{{\textit{W.h.p. }}}
\newcommand{\whp}{{\textit{w.h.p. }}}
\newcommand{\bin}{\textrm{Bin}}
\newcommand{\po}{\textrm{Po}}
\newcommand{\codeg}{\textrm{codeg}}
\newcommand{\E}{\mathrm{E}}
\newcommand{\V}{\mathrm{Var}}
\newcommand{\RT}{\textup{\textbf{RT}}}
\newcommand{\f}{\textup{\textbf{f}}}
\newcommand{\z}{\textrm{\textbf{z}}}
\newcommand{\ex}{\textrm{\textbf{ex}}}
\newcommand{\sqbs}[1]{\left[ #1 \right]}
\newcommand{\of}[1]{\left( #1 \right)}
\newcommand{\bfrac}[2]{\of{\frac{#1}{#2}}}
\renewcommand{\l}{\ell}
\newcommand{\sm}{\setminus}
\newcommand{\mc}[1]{\mathcal{#1}}
\newcommand{\h}[1]{\mathbb{H}^{({#1})}}
\newcommand{\hh}{\mathbb{H}}
\newcommand{\g}{\mathbb{G}}

\title[]{Weak and strong versions of the 1-2-3 conjecture  for uniform hypergraphs}

\author{Patrick Bennett, Andrzej Dudek, Alan Frieze, Laars Helenius}
\address[Andrzej Dudek, Patrick Bennett and Laars Helenius]{Department of Mathematics, Western Michigan University, Kalamazoo, MI}
\email{\tt \{patrick.bennett,\;andrzej.dudek,\;laars.c.helenius\}@wmich.edu}
\address[Alan Frieze]{Department of Mathematical Sciences, Carnegie Mellon University, Pittsburgh, PA}
\email{\tt alan@random.math.cmu.edu}
\thanks{The second author was supported in part by Simons Foundation Grant \#244712 and by the National Security Agency under Grant Number H98230-15-1-0172. The United States Government is authorized to reproduce and distribute reprints notwithstanding any copyright notation hereon. The third author was supported in part by NSF grant DMS1362785.}

\begin{abstract}
Given an $r$-uniform hypergraph $H=(V,E)$ and a weight function $\omega:E\to\{1,\dots,w\}$, a coloring of vertices of~$H$, induced by~$\omega$, is defined by $c(v) = \sum_{e\ni v} w(e)$ for all $v\in V$. If there exists such a coloring that is strong (that means in each edge no color appears more than once), then we say that $H$ is strongly $w$-weighted. Similarly, if the coloring is weak (that means there is no monochromatic edge), then we say that $H$ is weakly $w$-weighted. In this paper, we show that almost all 3 or 4-uniform hypergraphs are strongly 2-weighted (but not 1-weighted) and almost all $5$-uniform hypergraphs are either 1 or 2 strongly weighted (with a nontrivial distribution). Furthermore, for $r\ge 6$ we show that almost all $r$-uniform hypergraphs are strongly 1-weighted. We complement these results by showing that almost all 3-uniform hypergraphs are weakly 2-weighted but not 1-weighted and for $r\ge 4$ almost all $r$-uniform hypergraphs are weakly 1-weighted.
These results extend a previous work of Addario-Berry, Dalal and Reed for graphs. We also prove general lower bounds and show that there are $r$-uniform hypergraphs which are not strongly $(r^2-r)$-weighted and not weakly 2-weighted. Finally, we show that determining whether a particular uniform hypergraph is strongly 2-weighted is NP-complete.
\end{abstract}

\date{\today}

\maketitle

\section{Introduction}
Let $H=(V,E)$ be a hypergraph on the vertex set~$V$ and with the set of (hyper)edges~$E\subseteq 2^V$. In this paper we consider \emph{$r$-uniform hypergraphs}, meaning that each edge has size $r$. Let $\omega:E\to\{1,\dots,w\}$ be a \emph{weight function}. Furthermore, let $c:V\to \mathbb{N}$ be a vertex-coloring induced by $\omega$ defined as $c(v) = \sum_{v \ni e} \omega(e)$
for each $v\in V$. The vertex-coloring is \emph{weak} if there is no monochromatic edge and \emph{strong} if each edge is rainbow, i.e., for each $\{v_1,\dots,v_r\}\in E$ we have $c(v_i)\neq c(v_j)$ for every $1\le i<j\le r$. We say that $H$ is \emph{weakly} \emph{$w$-weighted} if there exists $\omega:E\to\{1,\dots,w\}$ such that the vertex-coloring~$c$ induced by $\omega$ is weak. Similarly, we say that $H$ is \emph{strongly} \emph{$w$-weighted} if the corresponding coloring is strong.
Clearly each strongly $w$-weighted hypergraph is also weakly $w$-weighted. Note that for graphs ($r=2$) weak and strong colorings (and therefore weightings) are the same.

The well-known 1-2-3 Conjecture of Karo\'nski, \L{}uczak and Thomason~\cite{KLT} asserts that every graph without isolated edges is 3-weighted.
This conjecture attracted a lot of attention and has been studied by several researchers (see, e.g.,  a survey paper of Seamone~\cite{Seamone}). 
The 1-2-3 conjecture is still open but it is known due to a result of Kalkowski, Karo\'nski, and Pfender~\cite{5enough} that every graph without isolated edges is 5-weighted. (For previous results see  \cite{30enough, Gnp, 13enough}).

In this paper we study an analogue of this conjecture for hypergraphs. This stream of research was initiated by Kalkowski, Karo\'nski, and Pfender~\cite{KKP} who studied weakly weighted hypergraphs. In particular, they proved that any $r$-uniform hypergraph without isolated edges is weakly $(5r-5)$-weighted and that 3-uniform hypergraphs are even weakly 5-weighted. The authors also asked whether there is an absolute constant~$w_0$ such that every $r$-uniform hypergraph is weakly $w_0$-weighted. Furthermore, they conjectured that each 3-uniform hypergraph without isolated edges is weakly 3-weighted. We show that for almost all uniform hypergraphs these conjectures hold. 

In this paper we are interested in both strongly and weakly weighted hypergraphs. 
We say that a hypergraph is \emph{nice} if there is no pair of vertices $u$ and $v$ such that the set of edges containing $u$ is the same as the set of edges containing $v$. Note that only nice hypergraphs can be strongly weighted.

First we provide a lower bound on $w$ assuming that $r$-uniform hypergraph is strongly (or weakly) $w$-weighted.

\begin{theorem}\label{thm:lb}\ \\[-15pt]
\begin{enumerate}[label=(\roman*)]
\item\label{thm:lb:a} Let $r\ge 3$ be such that $r-1$ is prime power. Then, there exists a nice $r$-uniform hypergraph that is not strongly $(r^2-r)$-weighted. Furthermore, for all large $r$, there exists a nice $r$-uniform hypergraph that is not strongly $(r^2-o(r^2))$-weighted. 
\item\label{thm:lb:b} Let $r\ge 3$. Then, there exists an $r$-uniform hypergraph that is not weakly 2-weighted. 
\end{enumerate}
\end{theorem}

We say that \emph{almost all} $r$-uniform hypergraphs on $n$ vertices have property $\mathcal{P}$ if as $n$ tends to infinity, $o\big{(}2^{\binom{n}{r}}\big{)}$ $r$-uniform hypergraphs on $n$ vertices do not have property~$\mathcal{P}$. It is easy to see that almost all uniform hypergraphs are nice. Furthermore, we prove the following statement.

\begin{theorem}\label{thm:main}\ \\[-15pt]
\begin{enumerate}[label=(\roman*)]
\item\label{thm:main:a} Almost all $3$ or $4$-uniform hypergraphs are strongly 2-weighted (but not 1-weigh\-ted). 
\item\label{thm:main:b} The probability that a hypergraph chosen uniformly at random from the space of all 5-uniform hypergraphs of order $n$ is strongly 1-weighted is $e^{-\sqrt{6/\pi}}+o(1)$ and that it is strongly 2-weighted (but not 1-weighted) is $1-e^{-\sqrt{6/\pi}}+o(1)$.
\item\label{thm:main:c} Let $r\ge 6$.Then, almost all $r$-uniform hypergraphs are strongly 1-weighted.
\end{enumerate}
\end{theorem}
Part~\ref{thm:main:c} is much easier, since a hypergraph is strongly 1-weighted if and only if there is no edge containing a pair of vertices of the same degree. Thus, the interesting parts are~\ref{thm:main:a} and~\ref{thm:main:b}. In particular, the latter which gives a nontrivial distribution between 1 or 2-weightedness. Theorem~\ref{thm:main} can be viewed as an extension of a result of Addario-Berry, Dalal and Reed~\cite{Gnp} who showed that almost all graphs are 2-weighted. (Our proofs are different.)

We complement the above result with an analogous statement for weakly weightedness. 
\begin{theorem}\label{thm:main2}\ \\[-15pt]
\begin{enumerate}[label=(\roman*)]
\item\label{thm:main2:a} Almost all $3$-uniform hypergraphs are weakly 2-weighted (but not 1-weighted). 
\item\label{thm:main2:b} Let $r\ge 4$. Then, almost all $r$-uniform hypergraphs are weakly 1-weighted.
\end{enumerate}
\end{theorem}
This theorem mainly follows from Theorem~\ref{thm:main} since strong weightedness implies weak weightedness. 

We also show that determining whether a particular uniform hypergraph is strongly 2-weighted is NP-complete. Consequently, there is no simple characterization of strongly 2-weighted hypergraphs, unless P=NP. (Note that determining 1-weightedness in an $r$-uniform hypergraph can be done in polynomial time, since a hypergraph is 1-weighted if and only if all vertex degrees are distinct.) 
Formally, let
\[
\weighted{r} = \left\{H : H \text{ is a strongly 2-weighted $r$-uniform hypergraph}\right\}.
\]
\begin{theorem}\label{thm:np}
Let $r\ge 3$. Then $\weighted{r}$ is NP-complete.
\end{theorem}

Although we were unable to provide a general upper bound for strongly weighted hypergraphs we believe that the following holds.
\begin{conjecture}
For every $r\ge 3$ there is a constant $w=w(r)$ such that each nice $r$-uniform hypergraph is strongly $w$-weighted.
\end{conjecture}
\noindent
If true, then by Theorem~\ref{thm:lb} such constant $w$ is at least $r^2-r+1$ for infinitely many~$r$, and at least $r^2 - o(r^2)$ for all $r$. For weakly weighted hypergraphs we believe that the 1-2-3 conjecture for graphs also holds for hypergraphs.
\begin{conjecture}
For each $r\ge 3$ every $r$-uniform hypergraph without isolated edges is weakly $3$-weighted.
\end{conjecture}

\section{Lower bound}

Here we prove Theorem~\ref{thm:lb}. We start with part~\ref{thm:lb:a}.
First we recall some basic properties of projective planes (see, e.g.,~\cite{GG}).
A \textit{projective plane} $P(q)$ of order $q$ is an incidence structure on a set~${P}$ of points and a set~$L$ of lines such that: any two points lie in a unique line, and every line contains~$q+1$ points, and every point lies on~$q+1$ lines.
It is known that for every prime power~$q$ such incidence structure~$P(q)$ exists with~$|{P}|=|L|=q^2+q+1$.
In other words, $P(q)$ is a $(q+1)$-regular  $(q+1)$-uniform hypergraph of order $q^2+q+1$.

Suppose that a projective plane $(P,L)$ of order of~$q$ exists. We create a $2$-regular $(q+1)$-uniform hypergraph $H=(V,E)$ as follows. First we blow up each point $p\in P$ exactly $q+1$ times. Formally, 
\[
V = \{ (p,\ell) : p\in P, \ell\in L, p\in\ell \}.
\]
Clearly, $|V| = (q+1)(q^2+q+1)$. We define two types of edges of $H$. Let $p\in P$ and $\ell_1,\dots,\ell_{q+1}$ be lines incident with $p$. Then, $E_1$ consists of all edges of the form $\{(p,\ell_1), (p,\ell_2),\dots, (p,\ell_{q+1})\}$, which we denote by $e(p)$. Thus,
\[
E_1 = \{ e(p) = \{(p,\ell_1),\dots, (p,\ell_{q+1})\} : p\in P \text{ and } p \in \ell_i  \text{ for all } 1\le i\le q+1\}.
\]
Similarly, we define
\[
E_2 = \{ f(\ell)=\{(p_1,\ell),\dots,(p_{q+1},\ell)\} : \ell\in L \text{ and } p_i\in \ell \text{ for all } 1\le i\le q+1 \}.
\]
Set $E=E_1\cup E_2$. It is easy to see that $H$ is $2$-regular $(q+1)$-uniform hypergraph and nice.

We claim that $H$ is not strongly $(q^2+q)$-weighted. Assume for a contradiction that it is. Let $\omega:E\to\{1,\dots,q^2+q\}$ be such that the vertex-coloring $c$ induced by $\omega$ is strong. Since $|E_1|=q^2+q+1  > |\omega(E)|$, there are $e(p_1)$ and $e(p_2)$ in $E_1$ such that $\omega(e(p_1)) = \omega(e(p_2))$. By properties of $P(q)$ there is a line $\ell$ incident with $p_1$ and $p_2$. Thus,  $f(\ell)$ in $E_2$  contains $(p_1,\ell)$ and $(p_2,\ell)$. But
\[
c((p_1,\ell)) = \omega(e(p_1)) + \omega(f(\ell)) = \omega(e(p_2)) + \omega(f(\ell)) = c((p_2,\ell))
\]
and so $f(\ell)$ is not rainbow, a contradiction.

Set $r = q+1$. Clearly, $q^2+q = r^2-r$. So for each $r-1$ prime power there exists an $r$-uniform hypergraph which is not strongly  $(r^2-r)$-weighted.

The remaining part of~\ref{thm:lb:a} is very similar. It is well-known that for large $x$ there exists a prime number between $x(1-o(1))$ and $x$ (see, e.g.,~\cite{BHP}). Hence, for large $r$ there is a prime number $q$ such that $(r-1)(1-o(1))\le q\le r-1$. First, as above we start with $P(q)$ and construct a $2$-regular $(q+1)$-uniform hypergraph $H=(V,E)$ by blowing up each vertex of $P(q)$. If $q=r-1$, then we are done. Otherwise, if $q < r-1$, then we extend $H$ to an $r$-uniform hypergraph $I = (W,F)$. Let $W=V\cup U$, where $|U| = r+1$. For each $e\in E$ we define a new edge in~$F$ be adding to $f$ arbitrarily $r-(q+1)$ vertices from $U$. Finally, we add all possible $\binom{r+1}{r}$ edges on $U$. The resulting hypergraph $I$ is $r$-uniform, nice, and not strongly $(q^2+q)$-weighted. Since $q^2+q = r^2-o(r^2)$, the proof is finished.

\begin{figure}
\includegraphics[scale=0.9]{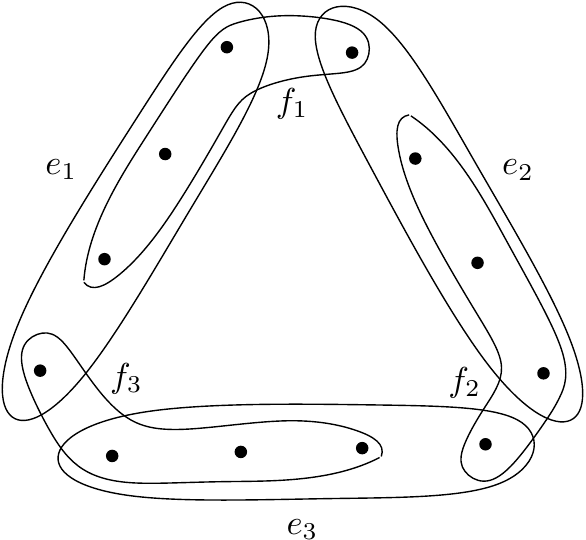}
\caption{A 4-uniform hypergraph which is not weakly 2-weighted.}
\label{fig:1}
\end{figure}

Now we prove~\ref{thm:lb:b}, which actually is an easy observation. Indeed, define an $r$-uniform hypergraph $H=(V,E)$ 
on $V=\{x_1,\dots,x_r,y_1,\dots,y_r,z_1,\dots,z_r\}$ with edges  $e_1 = \{x_1,\dots,x_r\}$, $e_2 = \{y_1,\dots,y_r\}$, $e_3 = \{z_1,\dots,z_r\}$, $f_1 = \{x_2,\dots,x_r,y_1\}$, $f_2 = \{y_2,\dots,y_r,z_1\}$, and $f_3 = \{z_2,\dots,z_r,x_1\}$ (see Figure~\ref{fig:1}). If there is a weak coloring of $H$ induced by some $w:E\to\{1,2\}$, then for some $i,j$ we must have $\omega(e_i) = \omega(e_j)$ and consequently edge $f_\ell \subseteq V(e_i) \cup V(e_j)$ is monochromatic, a contradiction.

\section{Tools used in the proof of Theorems~\ref{thm:main} and~\ref{thm:main2}}\label{sec:prelim}

In order to prove Theorems~\ref{thm:main} and ~\ref{thm:main2} we will consider random hypergraphs. Let $\h{k}(n,p)$ be an \emph{$r$-uniform random hypergraph} such that each of $\binom{[n]}{r}$ $r$-tuples is contained with probability~$p$, independently of others. We say that an event $E_n$ occurs \emph{with high probability}, or \whp for brevity, if $\lim_{n\rightarrow\infty}\Pr(E_n)=1$. We will use some standard probabilistic tools, which we state here for convenience (for more details see, e.g., \cite{AS,JLR}).

\begin{c2}
Let $X$ be a nonnegative integral random variable. If $\E(X) = o(1)$, then \whp $X=0$.
\end{c2}

\begin{c4}
Let $X$ be a nonnegative integral random variable. If $\V(X) = o(\E(X)^2)$, then \whp $X \ge 1$.
\end{c4} 

Let $\po(\lambda)$ denote the random variable with Poisson distribution with mean $\lambda$. Moreover, let $(X)_k :=X(X-1)(X-2) \cdots (X-k+1)$ denote the factorial moment of the random variable~$X$.
\begin{c5}
Let $\lambda$ be a positive constant. Suppose that $X_1,X_2,\dots$ are random variables such that for each fixed $k$ we have $\E((X_n)_k)\to\lambda^k$ as $n$ tends to infinity. Then, $X_n\sim \po(\lambda)$.
\end{c5}

\begin{ineq}\ \\[-15pt]
\begin{enumerate}[label=(\roman*)]
\item \emph{Markov's bound:} Let $X$ be a nonnegative integral random variable and $\gamma>0$. Then
\[
\Pr(X \ge \gamma) \le \E(X) / \gamma.
\]
\item \emph{Chernoff's bound:} Let $\bin(n,p)$ denote the random variable with binomial distribution with number of trials~$n$ and probability of success~$p$. If $X \sim \bin(n,p)$ and $0 < \gamma \leq \E(X)$, then
\[
\Pr \left( |X - \E(X)| \geq \gamma \right) \leq 2 \exp \left( -\gamma^2/(3\E(X)) \right) .
\]
\item \emph{Bernstein's bound:} Let $X_1, \ldots, X_m$ be independent random variables, and $X = \sum_{i=1}^m X_i$. Suppose that $|X_i - \E(X_i)| \le C$ always holds for all $i$, and $\gamma>0$. Then,
\[
\Pr \left( |X - \E(X)| \geq \gamma \right) \leq 2\exp \left ( -\frac{\frac{1}{2} \gamma^2}{\sum_{i=1}^m \V(X_j)+\tfrac{1}{3} C\gamma} \right ).
\]
\item \emph{Union bound:} If $E_1,\dots,E_m$ are events, then 
\[
\Pr \Big( \bigcup_{i=1}^m E_i \Big) \leq m \cdot \max \{\Pr(E_i): i \in [m]\}.
\]
\end{enumerate}
\end{ineq}

Several times we will also need to estimate binomial coefficients (for more details see, e.g., Chapter 22 in~\cite{Handbook}).

\newpage

\begin{bc}\ \\[-15pt]
\begin{enumerate}[label=(A\arabic*)]
\item\label{approx:2} Let $p>0$ and $\ell$ be functions of $m$ ($\ell$ can be negative). Assume that $\ell^2 = o(p)$ and $\ell^2 = o(m-p)$ as $m$ tends to infinity. Then,
\[
\binom{m}{p+\ell}\sim \binom{m}{p} \left(\frac{m-p}{p}\right)^\ell.
\]

\item\label{approx:1} Let $k\ge 1$ be a fixed integer. Then, 
\[
\sum_{i=0}^m \binom{m}{i}^k \sim \left( 2^m \sqrt{\frac{2}{\pi m}} \right)^k \sqrt{\frac{\pi m}{2k}}
= \Theta(2^{km} m^{-(k-1)/2}). 
\]

\item\label{approx:3} 
\[
\binom{m}{\lfloor m/2 \rfloor} \sim \frac{2^{m+1/2}}{\sqrt{\pi m}}.
\]
\end{enumerate}
\end{bc}

Throughout the next sections all logarithms are natural (base $e$) and all asymptotics are taken in $n$. 

We prove Theorem~\ref{thm:main} in Sections~\ref{sec:thm2:iii}-\ref{sec:thm2:ib} starting with the easiest case $r\ge 6$ and Theorem~\ref{thm:main2} in Sections~\ref{sec:thm3:i}-\ref{sec:thm3:ii}.

\section{For any $r\ge 6$ almost all $r$-uniform hypergraphs are strongly 1-weighted}\label{sec:thm2:iii}

Let $X_{2}^{(r)}$ count the number of pairs of vertices $\{u,v\}$ in $\h{r}(n,1/2)$ such that $\deg(u)=\deg(v)$. Observe that
\[
\Pr(\deg(u)=\deg(v)) = \sum_{a=0}^{\binom{n-2}{r-1}}  \binom{\binom{n-2}{r-1}}{a}^2 2^{-2\binom{n-2}{r-1}}
= \binom{2\binom{n-2}{r-1}}{\binom{n-2}{r-1}} 2^{-2\binom{n-2}{r-1}}.
\]
Thus, by~\ref{approx:3} 
\begin{equation}\label{eq:X2}
\E(X_2^{(r)}) = \binom{n}{2} \binom{2\binom{n-2}{r-1}}{\binom{n-2}{r-1}} 2^{-2\binom{n-2}{r-1}} 
\sim \frac{n^2}{2} \cdot \frac{2^{2\binom{n-2}{r-1}+1/2}}{\sqrt{\pi 2\binom{n-2}{r-1}}} \cdot 2^{-2\binom{n-2}{r-1}}
\sim \frac{\sqrt{(r-1)!}}{2\sqrt{\pi}} n^{2-(r-1)/2} 
\end{equation}
which is $o(1)$ for any $r\ge 6$. Thus, the first moment method yields the statement.

\section{Almost all $5$-uniform hypergraphs are either 1 or 2 strongly weighted}\label{sec:thm2:ii}

By~\eqref{eq:X2}, $E(X_2^{(5)}) \sim \sqrt{6/\pi}$. Set $\lambda=\sqrt{6/\pi}$.
We will show by using the method of moments that $X_2^{(5)} \sim \po(\lambda)$. First we prove two auxiliary results which we will also use in the next sections.

\begin{lemma}\label{lem:2}
Let $r\ge 3$. Then, \whp each pair of vertices of $\h{r}(n,1/2)$ is contained in an edge.
\end{lemma}
\begin{proof} 
The probability that a fixed pair of vertices $u$ and $v$ is not contained in any edge is $2^{-\binom{n-2}{r-2}}= 2^{-\Omega(n^{r-2})}$. Thus, by the union bound we get that the probability that there exists a pair of vertices which is not contained in any edge is  at most $\binom{n}{2} 2^{-\Omega(n^{r-2})} = o(1)$.
\end{proof}

\begin{lemma}\label{lem:1}
Let $r\ge 3$ and $k_1,\dots,k_\alpha\ge2$ be integers. Let $k = k_1+\dots+k_\alpha$. Then, \whp
for each $\{v_{1,1},\dots,v_{1,k_1},v_{2,1},\dots,v_{2,k_2},\dots,v_{\alpha,1},\dots,v_{\alpha,k_\alpha}\} \subseteq [n]$ in $\h{r}(n,1/2)$,
\begin{align}
\label{lem:1:statement:1} \Pr\left( \bigcap_{i=1}^\alpha \deg(v_{i,1}) = \dots = \deg(v_{i,k_i}) \right) 
&\sim 2^{-k\binom{n-k}{r-1}} \prod_{i=1}^{\alpha} \sum_{a=0}^{\binom{n-k}{r-1}} \binom{\binom{n-k}{r-1}}{a}^{k_i}\\
\label{lem:1:statement:2}&= O(n^{-(r-1)(k-\alpha)/2}).
\end{align}
\end{lemma}

\begin{proof}
Let $U = \{v_{1,1},\dots,v_{1,k_1},v_{2,1},\dots,v_{2,k_2},\dots,v_{\alpha,1},\dots,v_{\alpha,k_\alpha}\}$ with $k=|U|$. For a fixed $S\subseteq U$ with $1\le |S|\le r$, let $x_S$ be the random variable that counts the number of edges containing $S$ and $r-|S|$ other vertices from $[n]\setminus U$. Clearly, $x_S\sim \bin\left( \binom{n-k}{r-|S|},1/2 \right)$.
The Chernoff bound together with the union bound yield that \whp for any $U$ and $S\subseteq U$,
\begin{equation}\label{eq:1}
\left|x_S - {\binom{n-k}{r-|S|}}/{2}\right| = O(n^{(r-|S|)/2}\sqrt{\log n}) = O(n^{(r-1)/2}\sqrt{\log n}).
\end{equation}
Let $v\in U$. Then,
\[
\deg(v) = \sum_{v\in S \subseteq U} x_S
\]
and the expected value is $\sum_{1\le s \le r}\binom{k-1}{s-1} {\binom{n-k}{r-s}}/{2}$.
Thus, \whp for any $U$ and $v\in U$,
\[
\left| \deg(v) - \sum_{1\le s < (r+1)/2}\binom{k-1}{s-1} {\binom{n-k}{r-s}}/{2}\right| = O(n^{(r-1)/2}\sqrt{\log n}).
\]
Conditioning on $x_S=\beta_S$ for $2\le |S|\le r$ and $S\subseteq U$, we get that the probability that $\deg(v) = a$ is 
\[
\binom{\binom{n-k}{r-1}}{a-\sum_{S\ni v}\beta_S} 2^{-\binom{n-k}{r-1}}.
\]
Due to~\eqref{eq:1} we may assume that $\beta_S = {\binom{n-k}{r-|S|}}/{2} + \ell_S$ and  $|a - \sum_{1\le s < (r+1)/2}\binom{k-1}{s-1} {\binom{n-k}{r-s}}/{2}| = O(n^{(r-1)/2}\sqrt{\log n})$, where $|\ell_S| = O(n^{(r-|S|)/2}\sqrt{\log n})=O(n^{(r-2)/2}\sqrt{\log n})$.

By~\ref{approx:2}, applied with $m=\binom{n-k}{r-1}$, $p=a-\sum_{S\ni v} {\binom{n-k}{r-|S|}}/{2}$, and $\ell = \sum_{S\ni v} \ell_S$, we obtain that
\begin{align*}
\binom{\binom{n-k}{r-1}}{a-\sum_{S\ni v}\beta_S}
&= \binom{\binom{n-k}{r-1}}{\big{(}a-\sum_{S\ni v} {\binom{n-k}{r-|S|}}/{2}\big{)}-\sum_{S\ni v} \ell_S}\\
& \sim \binom{\binom{n-k}{r-1}}{a-\sum_{S\ni v} {\binom{n-k}{r-|S|}}/{2}} \left( \frac{\binom{n-k}{r-1} - \left(a-\sum_{S\ni v} {\binom{n-k}{r-|S|}}/{2}\right)}{a-\sum_{S\ni v} {\binom{n-k}{r-|S|}}/{2}} \right)^{\sum_{S\ni v} \ell_S}.
\end{align*}
Now observe that 
\[
a-\sum_{S\ni v} {\binom{n-k}{r-|S|}}/{2}
= a - \sum_{2\le s \le r} \binom{k-1}{s-1} \binom{n-k}{r-s}/2
= \binom{n-k}{r-1}/2\;\pm\; O(n^{(r-1)/2}\sqrt{\log n}).
\]
This implies that
\[
\frac{\binom{n-k}{r-1} - \left(a-\sum_{S\ni v} {\binom{n-k}{r-|S|}}/{2}\right)}{a-\sum_{S\ni v} {\binom{n-k}{r-|S|}}/{2}} 
\le \frac{ {\binom{n-k}{r-1}}/{2} + O(n^{(r-1)/2}\sqrt{\log n}) }{  {\binom{n-k}{r-1}}/{2} - O(n^{(r-1)/2}\sqrt{\log n}) }
= 1 + O\left( \frac{\sqrt{\log n}}{n^{(r-1)/2}} \right)
\]
and
\[
\frac{\binom{n-k}{r-1} - \left(a-\sum_{S\ni v} {\binom{n-k}{r-|S|}}/{2}\right)}{a-\sum_{S\ni v} {\binom{n-k}{r-|S|}}/{2}}
\ge \frac{ {\binom{n-k}{r-1}}/{2} - O(n^{(r-1)/2}\sqrt{\log n}) }{  {\binom{n-k}{r-1}}/{2} + O(n^{(r-1)/2}\sqrt{\log n}) }
= 1 - O\left( \frac{\sqrt{\log n}}{n^{(r-1)/2}} \right).
\]
Hence, since $|\sum_{S\ni v_i} \ell_S| = O(n^{(r-2)/2}\sqrt{\log n})$, we get that 
\[
\left( \frac{\binom{n-k}{r-1} - \left(a-\sum_{S\ni v} {\binom{n-k}{r-|S|}}/{2}\right)}{a-\sum_{S\ni v} {\binom{n-k}{r-|S|}}/{2}} \right)^{\sum_{S\ni v} \ell_S} \sim 1,
\]
and consequently,
\[
\binom{\binom{n-k}{r-1}}{a-\sum_{S\ni v}\beta_S} \sim \binom{\binom{n-k}{r-1}}{a-\sum_{S\ni v} {\binom{n-k}{r-|S|}}/{2}}.
\]
Thus, conditioning on $\beta_S$'s for $S\subseteq U$ with $2\le |S| \le r$ the probability that $\deg(v_{i,1}) = \dots = \deg(v_{i,k_i})=a_i$ for each $1\le i\le \alpha$ equals 
\begin{align*}
2^{-k\binom{n-k}{r-1}} \cdot \prod_{i=1}^\alpha \prod_{j=1}^{k_i} \binom{\binom{n-k}{r-1}}{a_i-\sum_{S\ni v_{i,j}}\beta_S}
&= 2^{-k\binom{n-k}{r-1}} \cdot \prod_{i=1}^\alpha \binom{\binom{n-k}{r-1}}{a_i-\sum_{S\ni v_{i,1}}\beta_S}^{k_i} \\
&\sim 2^{-k\binom{n-k}{r-1}} \cdot \prod_{i=1}^\alpha \binom{\binom{n-k}{r-1}}{a_i-\sum_{S\ni v_{i,1}} {\binom{n-k}{r-|S|}}/{2}}^{k_i}
\end{align*}
and further the probability that $\deg(v_{i,1}) = \dots = \deg(v_{i,k_i})$ for each $1\le i\le \alpha$ (still conditioning on $\beta_S$'s) asymptotically equals
\begin{align*}
&2^{-k\binom{n-k}{r-1}} \cdot \sum_{a_1,\dots,a_{\alpha}} \binom{\binom{n-k}{r-1}}{a_1-\sum_{S\ni v_{1,1}} {\binom{n-k}{r-|S|}}/{2}}^{k_1} \cdot \ldots \cdot \binom{\binom{n-k}{r-1}}{a_\alpha-\sum_{S\ni v_{k_\alpha,1}} {\binom{n-k}{r-|S|}}/{2}}^{k_\alpha}\\
&=2^{-k\binom{n-k}{r-1}} \cdot \left(\sum_{a_1}\binom{\binom{n-k}{r-1}}{a_1-\sum_{S\ni v_{1,1}} {\binom{n-k}{r-|S|}}/{2}}^{k_1}\right) \cdot \ldots \cdot\left(\sum_{a_{\alpha}} \binom{\binom{n-k}{r-1}}{a_\alpha-\sum_{S\ni v_{k_\alpha,1}} {\binom{n-k}{r-|S|}}/{2}}^{k_\alpha}\right),
\end{align*}
where the summations are taken over all possible values of $a_1,\dots,a_{k}$ satisfying $|a_i - \sum_{1\le s < (r+1)/2}\binom{k-1}{s-1} {\binom{n-k}{r-s}}/{2}| = O(n^{(r-1)/2}\sqrt{\log n})$. Since
\[
\sum_{a_i}  \binom{\binom{n-k}{r-1}}{a_i-\sum_{S\ni v_{i,1}} {\binom{n-k}{r-|S|}}/{2}}^{k_i}
\sim \sum_{a=0}^{\binom{n-k}{r-1}}  \binom{\binom{n-k}{r-1}}{a}^{k_i},
\]
we get
\begin{multline*}
2^{-k\binom{n-k}{r-1}} \cdot \sum_{a_1,\dots,a_{\alpha}} \binom{\binom{n-k}{r-1}}{a_1-\sum_{S\ni v_{1,1}} {\binom{n-k}{r-|S|}}/{2}}^{k_1} \cdot \ldots \cdot \binom{\binom{n-k}{r-1}}{a_\alpha-\sum_{S\ni v_{k_\alpha,1}} {\binom{n-k}{r-|S|}}/{2}}^{k_\alpha}\\
\sim 2^{-k\binom{n-k}{r-1}} \prod_{i=1}^{\alpha} \sum_{a=0}^{\binom{n-k}{r-1}} \binom{\binom{n-k}{r-1}}{a}^{k_i},
\end{multline*}
and finally by the law of total probability,
\begin{align*}
\Pr\left( \bigcap_{i=1}^\alpha \deg(v_{i,1}) = \dots = \deg(v_{i,k_i}) \right)
\sim \sum_{\beta_S^*} \Pr\left( \bigcap_{S} y_S = \beta_S \right) \cdot 2^{-k\binom{n-k}{r-1}} \prod_{i=1}^{\alpha} \sum_{a=0}^{\binom{n-k}{r-1}} \binom{\binom{n-k}{r-1}}{a}^{k_i},
\end{align*}
where $\beta_S^*$ denotes the summation over all possible values of $\beta_S$ for each $2\le |S|\le r$ satisfying $|\beta_S - {\binom{n-k}{r-|S|}}/{2}| = O(n^{(r-|S|)/2}\sqrt{\log n})$. This completes the proof of \eqref{lem:1:statement:1}, since $\sum_{\beta_S^*} \Pr( \bigcap_{S} y_S = \beta_S) \sim 1$.

Now \eqref{lem:1:statement:2} easily follows from \ref{approx:1}. Indeed, since
\[
2^{-k_i\binom{n-k}{r-1}} \sum_{a=0}^{\binom{n-k}{r-1}} \binom{\binom{n-k}{r-1}}{a}^{k_i} = O(n^{-(r-1)(k_i-1)/2}),
\]
we obtain
\[
2^{-k\binom{n-k}{r-1}} \prod_{i=1}^{\alpha} \sum_{a=0}^{\binom{n-k}{r-1}} \binom{\binom{n-k}{r-1}}{a}^{k_i}
= O\left(\prod_{i=1}^{\alpha}n^{-(r-1)(k_i-1)/2}\right)
= O\left(n^{-(r-1)(k-\alpha)/2}\right).
\]
\end{proof}

Now we use Lemma~\ref{lem:1} to show that $\E((X_2^{(5)})_k)\to\lambda^k$ as $n$ tends to infinity. Observe that $(X_2^{(5)})_k$ consists of $\binom{n}{2}\binom{n-2}{2}\cdots\binom{n-2k+2}{2}$ terms of $k$ vertex-disjoint pairs and $O(n^{2k-1})$ remaining terms. The pairs in the remaining terms are not vertex disjoint. Let $U$ be the union over all vertices in such $k$ pairs. Clearly, $k\le |U|\le 2k-1$, where the lower bound, $k$, corresponds to a $k$-cycle. Vertices in $U$ can be divided into $\alpha$ groups (each of size $k_i \ge 2$) in such a way that in each group all vertices have the same degrees. Observe that $k_1+\dots+k_{\alpha} = |U|$ and since $|U|<2k$, $1\le \alpha < |U|/2$. Due to~\eqref{lem:1:statement:2} the probability of occurrence of this degree sequence is at most
$O(n^{-2(|U|-\alpha)}) =o\of{ n^{-|U|}}$. Thus, $O(n^{|U|} \cdot  n^{-2(|U|-\alpha)}) = o(1)$ and
\begin{align*}
\E\left((X_2^{(5)})_k\right) &\sim \binom{n}{2}\binom{n-2}{2}\cdots\binom{n-2k+2}{2} 2^{-2k\binom{n-2k}{4}} \prod_{i=1}^{k} \sum_{a=0}^{\binom{n-2k}{4}} \binom{\binom{n-2k}{4}}{a}^{2}\\
&= \binom{n}{2}\binom{n-2}{2}\cdots\binom{n-2k+2}{2} 2^{-2k\binom{n-2k}{4}} \binom{2\binom{n-2k}{4}}{\binom{n-2k}{4}}^{k}\\
&\sim \left( \E(X_2^{(5)}) \right)^k \sim \lambda^k,
\end{align*}
since by~\ref{approx:3}
\[
2^{-2\binom{n-2}{4}}  \binom{2\binom{n-2}{4}}{\binom{n-2}{4}} 
\sim 2^{-2\binom{n-2k}{4}} \binom{2\binom{n-2k}{4}}{\binom{n-2k}{4}}.
\]
Hence, the method of moments implies that $X_2^{(5)} \sim \po(\lambda)$ and consequently
\begin{equation}\label{eq:poisson:1}
\Pr(\h{5}(n,1/2) \text{ is strongly 1-weighted}) = \Pr(X_2^{(5)}=0)\sim e^{-\sqrt{6/\pi}}.
\end{equation}
It remains to show that  \whp the random hypergraph $\h{5}(n,1/2)$ is strongly 2-weighted. 

Let $X_{3}^{(r)}$ count the number of triples of vertices $\{v_1,v_2,v_3\}$ in $\h{r}(n,1/2)$ such that $\deg(v_1)=\deg(v_2)=\deg(v_3)$.
Then, by Lemma~\ref{lem:1} (applied with $k=k_1=3$ and $\alpha=1$)
\begin{equation}\label{eq:X3}
\E(X_3^{(r)}) \sim \binom{n}{3} 2^{-3\binom{n-3}{r-1}} \sum_{a=0}^{\binom{n-3}{r-1}} \binom{\binom{n-3}{r-1}}{a}^3 
= \Theta(n^{4-r}),
\end{equation}
which is $o(1)$ for $r=5$. 

So far we have established the following properties of $\h{5}(n,1/2)$. By the Markov bound, $\Pr(X_2^{(5)} \ge \log n) =o(1)$ and by~\eqref{eq:X3} and the first moment method \whp no three vertices have the same degree. 

Once the edges of $\h{5}(n,1/2)$ are revealed we may assume that we have $s$ vertex-disjoint pairs $\{u_i,v_i\}$ such that $\deg(u_i)=\deg(v_i)$, where $1\le i\le s\le \log n$. (All other vertices have different degrees.) Let $S=\bigcup_i\{u_i,v_i\}$. We show that 
there is a matching saturating all $u_i$'s such that each matching edge contains one vertex from $S$ and 4 vertices from $[n]\setminus S$. One can find such matching greedily. Assume that we already chose matching edges $e_1,\dots,e_k$ with $u_i\in e_i$. The number of edges incident to $u_{k+1}$ that are not disjoint with $e_1,\dots,e_k$ or $S\setminus\{u_{k+1}\}$ is at most $O(n^3\log n)$ but $\deg(u_i)=\Omega(n^4)$. Hence, we can extend the matching by a new edge containing $u_{k+1}$.

Finally, we are ready to define a 2-weighting. First we assign to each edge weight 2. Now the colors of all vertices are even and every vertex has a different color except the $u_i$ and $v_i$. Next we replace the weights of matching edges by 1 so now all vertices have different colors.

\section{Almost all 4-uniform hypergraphs are strongly 2-weighted (but not 1-weighted)}\label{sec:thm2:ia}

First we show using the second moment method that \whp $\h{4}(n,1/2)$ is not strongly 1-weighted. By~\eqref{eq:X2}, $\E(X_2^{(4)}) = \Theta(\sqrt{n})$
and clearly it goes to infinity together with~$n$. 

Now we compute the variance.
Let $X_2^{(4)} = \sum_{\{i,j\} \subseteq \binom{[n]}{2}} X_{i,j}$, where $X_{i,j}$ is an indicator random variable such that if $X_{i,j}=1$, then $\deg(v_i)=\deg(v_j)$. Thus,
\[
(X_2^{(4)})^2 = X_2^{(4)} + \sum_{\{i,j,k\}} X_{i,j} X_{j,k} + \sum_{\{i,j\} \cap \{k,\ell\}=\emptyset} X_{i,j} X_{k,\ell}.
\]
By Lemma~\ref{lem:1} (applied with $r=4$, $k=k_1=3$, and $\alpha=1$),
\[
\Pr(X_{i,j} X_{j,k} = 1) = \Pr(\deg(v_i)=\deg(v_j)=\deg(v_k)) = O(n^{-3}).
\]
Again by  Lemma~\ref{lem:1} (applied with $r=4$, $k=4$, $k_1=k_2=2$ and $\alpha=2$)
\begin{multline*}
\Pr(X_{i,j} X_{k,\ell} = 1) = \Pr(\deg(v_i)=\deg(v_j) \text{ and } \deg(v_k)=\deg(v_\ell))\\
\sim 2^{-4\binom{n-4}{3}} \left( \sum_{a=0}^{\binom{n-4}{3}} \binom{\binom{n-4}{3}}{a}^{2} \right)^2 
=2^{-4\binom{n-4}{3}} \binom{2\binom{n-4}{3}}{\binom{n-4}{3}}^2.
\end{multline*}
Thus,
\[
\sum_{\{i,j,k\}} \E(X_{i,j} X_{j,k}) = O(n^3 \cdot n^{-3}) = O(1)
\]
and
\begin{align*}
\sum_{\{i,j\} \cap \{k,\ell\}=\emptyset} \E(X_{i,j} X_{k,\ell})
\sim \binom{n}{2}\binom{n-2}{2} 2^{-4\binom{n-4}{3}} \binom{2\binom{n-4}{3}}{\binom{n-4}{3}}^2
\sim \E(X_2^{(4)})^2,
\end{align*}
where the last asymptotic equality follows from~\ref{approx:3}. 
Consequently, $\E((X_2^{(4)})^2)\sim \E(X_2^{(4)}) + \E(X_2^{(4)})^2$ and hence $\V(X_2^{(4)})\sim \E(X_2^{(4)}) = o(\E(X_2^{(4)})^2)$. Thus, the second moment method together with Lemma~\ref{lem:2} implies that \whp there is a pair of vertices of the same degrees which is contained in a edge (so the hypergraph is not strongly 1-weighted).

Now we show that \whp $\h{4}(n,1/2)$ is strongly 2-weighted. We already observed that $\E(X_2^{(4)}) = O(\sqrt{n})$. Thus, by the Markov bound $\Pr(X_2^{(4)} \ge \sqrt{n}\log n) = o(1)$. By~\eqref{eq:X3}, $\E(X_3^{(4)})=O(1)$. Thus, again the Markov bound implies that $\Pr(X_3^{(4)} \ge \log n) = o(1)$. 
Lemma~\ref{lem:1} implies that the probability that vertices $v_1,v_2,v_3$ and $v_4$ in $\h{4}(n,1/2)$ have the same degrees is asymptotically equal to
\[
2^{-4\binom{n-4}{3}} \sum_{a=0}^{\binom{n-4}{3}} \binom{\binom{n-4}{3}}{a}^4 = O(n^{-9/2}).
\]
Hence, by the union bound (taken over all $\binom{n}{4}$ quadruples) we get that \whp there are no such four vertices. Similarly, 
(with essentially the same proof as Lemma~\ref{lem:1}) we get that \whp there are no four vertices $v_1,v_2,v_3$ and $v_4$ in $\h{4}(n,1/2)$ such that $\deg(v_1)=\deg(v_2)=\deg(v_3)=\deg(v_4)-1$.

Once the edges of $\h{4}(n,1/2)$ are revealed we may assume that there are $a$ vertex-disjoint pairs $\{u_i,v_i\}$ and $b$ triples $\{x_j,y_j,z_j\}$ such that  $\deg(u_i) = \deg(v_i)$, $\deg(x_j) = \deg(y_j) = \deg(z_j)$, and $1\le i \le a\le \sqrt{n}\log n$ and $1\le j\le b \le \log n$. (All other vertices have different degrees.)
Let $S_1 = \bigcup_i \{u_i,v_i\}$ and $S_2=\bigcup_j \{x_j,y_j,z_j\}$ and $T = [n] \setminus (S_1\cup S_2)$. We show that there are edges $e_i$, $f_j$, $f'_j$, and $g_j$ such that $u_i\in e_i$, $\{x_j\}= f_j\cap f'_j$, $y_j\in g_j$ and $e_i$, $f_j$, $f'_j$, and $g_j$  contain no other vertices from $S_1\cup S_2$ and $e_i\cap T$, $f_j\cap T$, $f'_j\cap T$, and $g_j\cap T$ are pairwise vertex-disjoint. 
Similarly as in the previous section one can find such edges greedily. First assume that we already found $e_1,\dots,e_k$ edges so that $u_i\in e_i$. The number of edges incident to $u_{k+1}$ that are not disjoint with $e_1,\dots,e_k$ or $S_1\cup S_2 \setminus\{u_{k+1}\}$ is at most $O(n^2\sqrt{n}\log n)$ but $\deg(u_i)=\Omega(n^3)$. Hence, we can extend $e_1,\dots,e_{k}$ by a new edge $e_{k+1}$ that contains~$u_{k+1}$.
Similarly we find edges $f_j$, $f'_j$, and $g_j$. 

Now we define a 2-weighting. First we assign to each edge weight 2. Now the colors of all vertices are even and only vertices $\{u_i,v_i\}\subseteq S_1$ and $\{x_j,y_j,z_j\}\subseteq S_2$ have the same color. Furthermore, there is no vertex $w$ such that for some $j$ we have $\deg(x_j) = \deg(y_j) = \deg(z_j) = \deg(w)-1$ or $\deg(x_j) = \deg(y_j) = \deg(z_j) = \deg(w)$.  Next we replace the weights of all  $e_i$, $f_j$, $f'_j$, and $g_j$ edges by 1 yielding all vertices to have different colors.

\section{Almost all 3-uniform hypergraphs are strongly 2-weighted (but not 1-weighted)}\label{sec:thm2:ib}

As in Section~\ref{sec:thm2:ia} the second moment method and Lemma~\ref{lem:2} imply that \whp there is a pair of vertices of the same degrees which is contained in a edge. Thus, \whp $\h{3}(n,1/2)$ is not strongly 1-weighted. As a matter of fact we will see later that \whp $\h{3}(n,1/2)$ is not even weakly 1-weighted.

Now we show that \whp $\h{3}(n,1/2)$ is strongly 2-weighted. Here our proof method differs from our other proofs that hypergraphs are 2-weighted. Since the expected number of pairs of vertices with the same degree is linear, we cannot simply give all edges weight 2 and then alter the weighting by flipping the weights of a few edges. We use a more complicated argument which we will outline after some lemmas.

First we need some auxiliary results. For $G=(V,E)$ and $S\subseteq V$, let $N(S)$ denote the \emph{neighborhood of $S$}, i.e., the set of all vertices in $V$ adjacent to some element of~$S$.

\begin{lemma}\label{lem:bipartite}
There exists a positive constant $\gamma$ such that with probability $1-o(1/n)$ the random bipartite graph $\g(n,n,1/2)$ contains at least $\gamma n$ edge disjoint perfect matchings.
\end{lemma}
This lemma is a weaker version of a more general result of Frieze and Krivelevich~\cite{FK} (where the authors obtained an optimal constant $\gamma=1/2-o(1)$). For the sake of completeness we show here a simple proof.

\begin{proof}
Let $G=\g(n,n,1/2)$ be a random bipartite graph on the set of vertices $A\cup B$, where $|A| = |B| = n$. Set $\gamma=1/10$. 

First observe that since for any $u,v \in A$ and $u,v \in B$ we have $|N(v)|\sim \bin(n,1/2)$ and $|N(\{u,v\})|\sim\bin(n,3/4)$, the Chernoff bound yields that with probability $1-o(1/n)$ for any $u,v \in A$ and $u,v \in B$,
\[
|N(v)| \ge n/2 - O(\sqrt{n\log n}) \quad \text{ and }\quad |N(\{u,v\})| \ge 3n/4 - O(\sqrt{n\log n}).
\]

Assume that we already found a collection $\mathcal{M}_i = \{M_1,\dots,M_{i}\}$ of perfect matchings in $G$. We show that $G_{i+1} = G\setminus \mathcal{M}_i$ also contains a perfect matching $M_{i+1}$. It suffices to show that if $i<\gamma n$, then the Hall condition holds, i.e., 
\begin{equation}\label{eq:hallS}
\text{if } S\subseteq A \text{ and } |S| \le n/2, \text{ then } |N_{G_{i+1}}(S)| \ge |S|,
\end{equation}
and
\begin{equation}\label{eq:hallT}
\text{if } T\subseteq B \text{ and } |T| \le n/2, \text{ then } |N_{G_{i+1}}(T)| \ge |T|.
\end{equation}
Indeed, if $S=\{v\}$, then 
\[
|N_{G_{i+1}}(S)| = |N_{G_{i+1}}(v)| = |N_G(v)| - i \ge n/2 - O(\sqrt{n\log n}) - \gamma n \ge 1 = |S|.
\] 
Therefore, we may assume that $2 \le |S| \le n/2$. Let $\{u,v\} \subseteq S$. Then, 
\[
|N_{G_{i+1}}(S)| \ge |N_{G_{i+1}}(\{u,v\})| \ge |N_{G}(\{u,v\})| -2i
\ge 3n/4 - O(\sqrt{n\log n}) -2\gamma n \ge n/2 \ge |S|
\]
and \eqref{eq:hallS} holds. Similarly, \eqref{eq:hallT} holds, too.
\end{proof}

\begin{lemma}\label{lem:hyper_matchings}
Let $\hh$ be a 3-partite 3-uniform random hypergraph on vertex set $V_1\cup V_2\cup V_3$ with $|V_1| = |V_2| = |V_3| = n$ and probability $1/2$. Then, there exists a positive constant~$\gamma$ such that \whp $\hh$ contains at least $\gamma n^2$ edge disjoint perfect matchings. 
\end{lemma}
\begin{proof}
Consider first  a complete bipartite graph $F$ on $V_1\cup V_2$. Then, since $F$ is $n$-regular bipartite graph, it can be decomposed into $n$ edge disjoint perfect matchings, say $F = M_1\cup\dots\cup M_n$. Let $H_i$ be a 3-partite 3-uniform hypergraph on $V_1\cup V_2 \cup V_3$ and the set of edges
\[
E_i = \{ e \cup v : e\in M_i \text{ and } v\in V_3 \}.
\]
Thus, the complete 3-partite 3-uniform hypergraph on $V_1\cup V_2 \cup V_3$ is the edge disjoint union of $H_i$'s over all $1\le i \le n$. We generate a random 3-partite 3-uniform hypergraph by revealing edges in each $H_i$. It suffices to show that the random hypergraph induced by $H_i$ contains with probability $1-o(1/n)$ at least $\gamma n$ edge disjoint perfect matchings, where $\gamma$ is a constant from Lemma~\ref{lem:bipartite}. The latter is obviously true since the random graph induced by $H_i$ can be viewed as $G=\g(n,n,1/2)$ on $M \cup V_3$ (since $|M| = |V_3| = n$) and any perfect matching in $G$ yields a perfect matching in the 3-partite hypergraph.
\end{proof}

Now we are ready to show that \whp $\h{3}(n,1/2)$ is strongly 2-weighted.
\begin{proof}

We start with an outline. First we consider an equipartition $V_1 \cup V_2 \cup V_3$ of the vertex set $V=[n]$.  Then we reveal the random hypergraph $\h{3}(n,1/2)$. Due to Lemma~\ref{lem:hyper_matchings} (assuming that $n$ is divisible by 3) there is \whp a family $\mathcal{M}$ of disjoint perfect matchings such that $\mathcal{M} = \lfloor \gamma  n^2 \rfloor$ and each edge in each matching has one vertex in each part of the partition. We then randomly label the vertices in $V_1$ with real number labels in $[0, 1/9]$, the vertices in $V_2$ with labels in $[1/9,  2/9]$, and the vertices in $V_3$ with labels in $[2/9,  1/3]$. Let $\l(v)$ denote the label of vertex $v$. We then randomly weight the edges, where any non-matching edge $\{u,v,w\} \in \binom{[n]}{3}$ gets weight $2$ with probability $\l(u)+\l(v)+\l(w)$, and any matching edge gets weight $2$ with probability $\frac 12$. 
We will show that 
\begin{align*}
\E(|\{c: \; 0 \le c \le n^2&, \; \exists\, x, y \mbox{ such that } |c(x)-c| \le 2 \text{ and } |c(y)-c| \le 2\} |)\\
& \le \sum_{c=0}^{n^2}\E(\left|\left\{ \{x, y\}:  |c(x)-c| \le 2 \text{ and } |c(y)-c| \le 2 \right\}\right|) = O(\log^4 n).
\end{align*}
We will call pairs $\{x, y\}$ such that  $ |c(x)-c| \le 2$ and  $|c(y)-c| \le 2$ for some $c$ \emph{dangerous} and pairs such that $c(x) = c(y)$ \emph{bad}.  We will also show that 
\begin{align*}
\E(|\{c: \; 0 \le c \le n^2, \; &\exists\, x, y, z \mbox{ such that } |c(x)-c|\le 2, |c(y)-c|\le 2, |c(z)-c| \le 2\} |)
\end{align*}
is $o(1)$.
Finally, we will alter the weights on certain edges so that we kill all bad pairs (i.e. after changing the weights no two vertices will have the same color) which will complete the proof.

In this paragraph and the next we are revealing only the random hypergraph $\hh = \h{3}(n,1/2)$, and all probabilities are with respect to this distribution. The degrees of vertices are concentrated by the Chernoff bound. More specifically, for any vertex $v$, and any distinct parts $V_i$ and  $V_j$ there are $m={n^2}/{9} +O(n)$ triples containing $v$ together with one additional vertex from $V_i$ and one from $V_j$ (the $O(n)$ term is to account for the possibility that $v$ is in one of $V_i$ or $V_j$). Each of these triples has probability $1/2$ of being an edge of $\hh$, independently, and so if we let $d_{i,j}(v)$ denote the number of these edges present in $\hh$, the Chernoff bound tells us that 
\[
\Pr(|d_{i,j}(v) - m/2| > n \log n) \le 2 \exp \left( -\frac{(n \log n)^2 }{3m/2} \right) =  \exp \left( - \Omega( \log^2 n )\right)
\]
and so the union bound gives us that the probability there exists $v, i, j$ such that $|d_{i,j}(v) - m/2| > n\log n$ is at most  $3n \cdot \exp \left\{ - \Omega( \log^2 n )\right\}= o(1)$.
Hence, \whp for every vertex $v$ and distinct parts $V_i, V_j$ there are ${n^2}/{18} + O(n \log n)$ edges containing $v$ and additionally one vertex from each of $V_i$, $V_j$. Similarly (using the Chernoff and union bounds), \whp for any vertex $v$ and part $V_i$ there are ${n^2}/{36} + O(n \log n)$ edges containing $v$ and additionally two vertices from $V_i$. 

Until the last paragraph of this proof will assume that $n$ is divisible by 3.
Due to Lemma~\ref{lem:hyper_matchings} there is \whp a family $\mathcal{M}$ of disjoint perfect matchings where $|\mathcal{M}|=\lfloor \gamma n^2\rfloor$ and each edge in each matching goes between the partition (i.e. has one vertex in each $V_i$). 

Henceforth we assume that all the edges of $\hh$ were revealed obtaining a hypergraph~$H$ that has the degree properties mentioned in the previous paragraphs, and the family of matchings $\mc{M}$. 

Next we reveal the vertex labels and the non-matching edge weights. The weight of edge $\{x, y, z\}$ is distributed as 
\[ w(x,y,z) = \begin{cases} 
      1 & \mbox{ with probability }   1 - \l(x)-\l(y)-\l(z) \\
      2 & \mbox{ with probability }  \l(x)+\l(y)+\l(z) 
   \end{cases}
\]
so 
\[
\E(w(x,y,z)\;|\;\l(x), \l(y), \l(z)) =  1 + \l(x)+\l(y)+\l(z)
\]
and therefore if by $c_{H \sm \mc{M}}(x)$ we denote the sum of the weights of non-matching edges containing $x$, and say we are given only the label $\l(x)$, and the hypergraph $H$ with family of matchings $\mc{M}$, and $x \in V_i$ then we have
\begin{align*}
\E(c_{H \sm \mc{M}}(x)\; |\; \l(x), H) &= \sum_{\{x, v, w \}\in E(H) \sm \mc{M}} \E(w(x,v,w)\;|\;\l(x))\\
& = \sum_{\{x, v, w \}\in E(H)\sm \mc{M}} (1 + \l(x)+\E(\l(v))+\E(\l(w)))\\
&= \of{1+ \l(x)} deg_{H \sm \mc{M}}(x) + \sum_{\{x, v, w \}\in E(H)\sm \mc{M}} (\E(\l(v))+\E(\l(w)) )\\
&=  \l(x) \of{\frac 14 - \gamma}n^2   + \Theta(n^2),
\end{align*}
where on the last line we have used our estimate of $\deg_{H \sm \mc{M}}(x)$. Note that the $\Theta(n^2)$ term may depend on $i$ (recall $x \in V_i$) but does not otherwise depend on $x$ or $\l(x)$ (by the fact that the degree of each vertex $v$ into sets $V_i$, $V_j$ etc. is concentrated).

Now since $c_{H \sm \mc{M}}(x) =  \sum_{\{x, v, w \}\in E(H) \sm \mc{M}} w(x, v, w)$, and the random variables $w(x,v,w)$ are independent (given the vertex labels), we can apply Bernstein's inequality. For our application we use $m =deg_{H \sm \mc{M}}(x)$. We can easily use $C=2$, and put $\V(w(x, v, w)) \le \E(w(x, v, w)^2) \le 4$. We will set $\gamma = n \log n$. Then we get 
\begin{align*}
\Pr \left (\left|c_{H \sm \mc{M}}(x)  -\E\left(c_{H \sm \mc{M}}(x)\;|\;\l(x), H\right) \right| > n \log n \right ) 
&\leq 2\exp \left ( -\frac{\tfrac{1}{2} n^2 \log^2 n}{4\deg_{H \sm \mc{M}}(x)+\tfrac{1}{3} \cdot  2 n \log n} \right )\\
& = \exp\left\{ -\Omega \left(\log^2 n \right) \right\}
\end{align*}
and so by the union bound, \whp for each vertex $x$, $c_{H \sm \mc{M}}(x)$ is within $n \log n$ of its expectation. 

Now for any fixed integer $a$ and fixed vertex $x$, 
\[
Pr\left(\left|\l(x) \of{\frac 14 - \gamma}n^2 - a\right|\le  n \log n\right) = O(\log n /n),
\]
since this is the probability that $\l(x)$ falls within an interval of length $O(\log n /n)$. Since the vertex labels are independent, the probability that there are $\log^2 n$ many vertices $x$ with $\left|\l(x) \of{\frac 14 - \gamma}n^2 - a\right|\le  n \log n$ is at most 

\begin{align*}
\binom{n}{\log^2 n} \cdot \of{O(\log n /n)}^{\log^2 n} &\le \bfrac{ne}{\log^2 n}^{\log^2 n}\of{O(\log n /n)}^{\log^2 n}\\
& = \of{O\of{\frac{1}{\log n}}}^{\log^2 n} =o(n^2).
\end{align*}
Therefore, by the union bound over integers $c$ from $0$ to $n^2$, we have that for all such $c$  the number of vertices $x$ with $|c_{H \sm \mc{M}}(x) - c| \le n \log n$ is at most $ \log^2 n$. Henceforth we assume that the labels and non-matching edge weights have been revealed and they satisfy this property. 

Now we reveal the matching edge weights. These are $1$ or $2$ with probability $1/2$. We will denote by $c_\mc{M}(x)$ the sum of the weights of matching edges adjacent to $x$; note that $c(x) = c_\mc{M}(x) + c_{H \sm \mc{M}}(x)$. The key fact we will use here is that since $|\mc{M}| = \lfloor \gamma n^2\rfloor$, $c_\mc{M}(x)$ is not likely to be any one particular value. Indeed, since $c_\mc{M}(x)-|\mc{M}| \sim \bin(|\mc{M}|, 1/2)$,
the mode of $c_\mc{M}(x)$ occurs with probability 
\[
\binom{|\mc{M}|}{\lfloor |\mc{M}|/2 \rfloor} 2^{-|\mc{M}|} \sim \frac{\sqrt{2/\gamma \pi}}{ n }= O(1/n),
\]
where the latter follows from~\ref{approx:3}.
Also, \whp for all $x$ we have $|c_\mc{M}(x) - \frac 32 \gamma n^2| \le n \log n$. Therefore, 
\begin{multline*}
\sum_{c=0}^{n^2}\E\left(\left|\left\{ \{x, y\}: |c(x)-c|\le 2 \text{ and }  |c(y)-c| \le 2  \right\}\right|\right) \\
= O\of{n^2 \cdot (\log^2 n )^2 \cdot (1/n)^2} = O(\log^4 n),
\end{multline*}
where on the last line we get the $(1/n)^2$ in the big-$O$ term since the probability that $c(x)=c$ is $O(1/n)$, and conditioning on that event, the event that $c(y)=c$ still has probability $O(1/n)$ (since revealing $c(x)$ only reveals the weights of $O(n)$ of the $\lfloor \gamma n^2\rfloor$ matching edges containing $y$, $c(y)$ still essentially has the same distribution as it did before conditioning) and similarly for the conditional probability that $c(z)=c$. Similarly,
\begin{multline*}
\E(|\{c: \; 0 \le c \le n^2, \; \exists\, x, y, z \mbox{ such that } |c(x)-c|\le 2,  |c(y)-c|\le 2, |c(z)-c| \le 2\} \\
\le O\of{n^2 \cdot (\log^2 n )^3 \cdot (1/n)^3} = o(1).
\end{multline*}

Thus, by the Markov bound \whp there are at most $\log^5 n$ dangerous pairs, and there is no triple $x, y, z$ such that  $|c(x)-c|\le 2$, $|c(y)-c|\le 2$, and $|c(z)-c| \le 2$ for any $c$. Suppose there are $b \le \log^5 n$ bad pairs and let  $\{\{u_i,v_i\}: 1 \le i \le b\}$ be the set of  bad pairs. For each $i$ we will choose an edge $e_i$ such that $u_i\in e_i$ and $e_i$ does not intersect any other dangerous pair and all edges $e_i$ are vertex-disjoint. As in the previous sections this is easy since each vertex has degree $\Theta(n^2)$, the number of edges containing any two vertices is $O(n)$, and there are at most $\log^5 n$ dangerous pairs. Now for each bad pair $\{u_i,v_i\}$ we flip the weight of $e_i$ (from 1 to 2 or from 2 to 1). Thus $u_i$ no longer has the same color as $v_i$. Also note that now no two vertices in the entire hypergraph can have the same color, since the new color of any vertex can only differ from its old color by at most 1, the only vertices in dangerous pairs whose colors are changed are the vertices that are in bad pairs (and they are only changed to make them not bad anymore), and there is no triple $x, y, z$ such that $|c(x)-c|\le 2$,  $|c(y)-c|\le 2$, and $|c(z)-c| \le 2$ for any $c$. This completes the proof assuming that 3 divides $n$.

If $n$ is not divisible by 3, then we can still use an equipartition $V_1$, $V_2$, $V_3$, and a family of (not quite perfect) matchings $\mc{M}$ where $|\mc{M}| = \gamma n^2$, each $M \in \mc{M}$ has size $\lfloor n/3 \rfloor$, and each vertex is covered by $\gamma n^2 - O(n)$ matchings in $\mc{M}$. It is relatively straightforward to check that with this small change, the same proof still works.
\end{proof}

\section{For any $r\ge 4$ almost all $r$-uniform hypergraphs are weakly 1-weighted}\label{sec:thm3:i}

Now let $X_{4}^{(r)}$ counts the number of quadruples of vertices $\{v_1,v_2,v_3,v_4\}$ in $\h{r}(n,1/2)$ such that $\deg(v_1)=\deg(v_2)=\deg(v_3)=\deg(v_4)$.
Then, by Lemma~\ref{lem:1} (applied with $k=k_1=4$ and $\alpha=1$)
\begin{equation}
\E(X_4^{(r)}) \sim \binom{n}{4} 2^{-4\binom{n-4}{r-1}} \sum_{a=0}^{\binom{n-4}{r-1}} \binom{\binom{n-4}{r-1}}{a}^4
= \Theta(n^{4-3(r-1)/2}),
\end{equation}
which is $o(1)$ for $r\ge 4$. 

\section{Almost all 3-uniform hypergraphs are weakly 2-weighted (but not 1-weighted)}\label{sec:thm3:ii}

By Theorem~\ref{thm:main}\ref{thm:main:a} \whp $\h{r}(n,1/2)$ is strongly 2-weighted which obviously implies that it is also weakly 2-weighted. Hence, we only need to show that \whp $\h{r}(n,1/2)$ is not weakly 1-weighted. We do it by applying the second moment method as in Section~\ref{sec:thm2:ia}.

Let $X_{3}$ counts the number of edges $\{v_1,v_2,v_3\}$ in $\h{3}(n,1/2)$ such that $\deg(v_1)=\deg(v_2)=\deg(v_3)$. By Lemma~\ref{lem:1} we get
\[
\Pr(\deg(v_1) = \deg(v_2) = \deg(v_3), \{v_1, v_2, v_3\} \in E) 
\sim \frac{1}{2}\sum_{a=0}^{\binom{n-3}{2}} \binom{\binom{n-3}{2}}{a}^3 2^{-3\binom{n-3}{2}}
\sim \frac{2}{\pi \sqrt{3} n^2 },
\]
where the latter follows from~\ref{approx:1} (applied with $m=\binom{n-3}{2}$ and $k=3$). Thus,
\[
\E(X_3) \sim \binom{n}{3} \frac{2}{\pi \sqrt{3} n^2 },
\]
which goes to infinity together with~$n$.

We show that $\E(X_3^2) \sim (\E(X_3))^2$. For $e\in \binom{[n]}{3}$,
let $X_e$ be an indicator random variable which is equal to 1 if all three vertices in $e$ have the same degree.
Thus,
\[
X_3^2 = X_3 + \sum_{|e\cap f| = 0} X_{e} X_{f} + \sum_{|e\cap f| = 1} X_{e} X_{f} + \sum_{|e\cap f| = 2} X_{e} X_{f}.
\]

If $|e\cap f|=2$, then we may assume that $e\cup f = \{v_1,\dots,v_4\}$ and by Lemma~\ref{lem:1}
\begin{align*}
\Pr(X_eX_f=1) &= \Pr(\deg(v_1)=\dots=\deg(v_4) \mbox{ and } e,f\in E)\\
&\sim \frac{1}{4}\sum_{a} \binom{\binom{n-4}{2}}{a}^4 2^{-4\binom{n-4}{2}}
=O\left( \frac{1}{n^3} \right)
\end{align*}
and consequently,
\[
\sum_{|e\cap f| = 2} \E(X_{e} X_{f}) = O\left( n^4 \cdot \frac{1}{n^3} \right)=  O(n) = o((\E(X_3))^2).
\]
Similarly,
\[
\sum_{|e\cap f| = 1} \E(X_{e} X_{f}) = O\left( n^5 \cdot \frac{1}{n^4} \right)=  O(n) = o((\E(X_3))^2).
\]
Finally, if $|e\cap f|=0$, then Lemma~\ref{lem:1} (applied with $r=3$, $k=6$, $k_1=k_2=3$, and $\alpha=2$) implies that
\[
\sum_{|e\cap f| = 0} \E(X_{e} X_{f}) \sim \frac{1}{4}\binom{n}{3} \binom{n-3}{3} 2^{-5\binom{n-6}{2}} \left( \sum_{a} \binom{\binom{n-6}{2}}{a}^3 \right)^2 \sim (\E(X_3))^2.
\]
Thus, we are done by the second moment method.


\section{NP-completeness of $\weighted{r}$}

Let $r\ge 3$. First note that $\weighted{r}$ is clearly in NP, since for a given hypergraph $H=(V,E)$ and $\omega:E\to\{1,2\}$ one can verify in polynomial time whether a vertex-coloring induced by $\omega$ is strong.

It is known that for graphs $\weighted{2}$ is NP-complete as it was proven independently by Dehghan, Sadeghi and Ahadi~\cite{NPcomplete2}, and Dudek and Wajc~\cite{NPcomplete}. 
 
In order to prove that $\weighted{r}$ is NP-complete, we show a reduction from $\weighted{2}$ to $\weighted{r}$. To this end, we define a polynomial time reduction~$h$, such that $G\in \weighted{2}$ if and only if $h(G)\in\weighted{r}$.

We will need an auxiliary gadget. First we define an $r$-partite $r$-uniform hypergraph $T=(V,E)$. Let $V=V_1\cup\dots\cup V_r$, where $|V_i|=i$ for each $1\le i\le r$, and $E$ be the set of edges consisting of all possible edges containing exactly one vertex from each~$V_i$. Observe that if $v\in V_i$, then $\deg(v) = r! / i$. Consequently $T$ is strongly 1-weighted (and so nice). We refer to the unique vertex $v\in V_1$ as a \emph{root}.
Let $T(k)$ be a union of $k$ copies of $T$ with the same root which we refer as the root of $T(k)$ (any two copies only share its roots). Clearly, 
$T(k)$ is still strongly 1-weighted and if $v$ is its root, then $\deg_{T(k)}(v) = k\cdot r!$.

Now we are ready to show a reduction from $\weighted{2}$ to $\weighted{r}$, $h$, such that $G\in \weighted{2}$ if and only if $h(G)\in\weighted{r}$.
Let $G=(V,E)$ be a graph of order~$n$. We construct a nice $r$-uniform hypergraph $h(G)=(W,F)$ as follows. For each edge $e=\{x,y\}$ in $G$ we define an edge $\{x,y,v_1,\dots v_{r-2}\}$ in $h(G)$, where all $v_i$'s a different for all $e$ and $i$. Now to each $v_i$ (for $1\le i\le r-2$) we attach a copy of $T(2in)$ on the new set of vertices with $v_i$ as its root. 

Let us assume that $G=(V,E )\in \weighted{2}$. Thus, there is $\omega_G:E\to\{1,2\}$ such that the vertex-coloring $c_G$ induced by $\omega_G$ is proper. Let $H=h(G)=(W,F)$. We define a weight function $\omega_H:F\to\{1,2\}$ as follows. To each edge $\{x,y,v_1,\dots v_{r-2}\}\in F$ derived from $\{x,y\} \in E$, we assign $\omega_G(\{x,y\})$; otherwise we assign weight~1. Now we claim that the vertex-coloring $c_H:W\to\mathbb{N}$ induced by~$\omega_H$ is strong. By construction, any edge of~$H$ which is contained in a copy of $T(2in)$ is rainbow. We show that this also holds for $\{x,y,v_1,\dots v_{r-2}\}\in F$ derived from $G$. Observe that 
\[
c_H(v_i) = c_G(v_i) + 2in \cdot r!,
\]
and consequently, $2n < c_H(v_1) < \dots < c_H(v_{r-2})$. Moreover, since $c_G(x)\neq c_G(y) \le 2(n-1)$, we get that $\{x,y,v_1,\dots v_{r-2}\}$ is rainbow. Hence, $h(G)\in \weighted{r}$.

Now suppose that $G=(V,E )\notin \weighted{2}$. We show that $H=h(G)=(W,F)\notin \weighted{r}$. Assume not. That means there exists a weight function $\omega_H:W\to\{1,2\}$ such that the vertex-coloring $c_H$ induced by $\omega_H$ is strong. Now let $\omega_G:V\to\{1,2\}$ be such that $\omega_G(\{x,y\}) = \omega_H(\{x,y,v_1,\dots v_{r-2}\})$. Since $c_G(x) = c_H(x)$ and $c_H$ is strong, we conclude that $c_G$ is proper, a contradiction.


\begin{thebibliography}{10}

\bibitem{30enough}
L.~Addario-Berry, K.~Dalal, C.~McDiarmid, B.~A. Reed, and A.~Thomason,
  \emph{Vertex-colouring edge-weightings}, Combinatorica \textbf{27} (2007),
  no.~1, 1--12.

\bibitem{Gnp}
L.~Addario-Berry, K.~Dalal, and B.~A. Reed, \emph{Degree constrained
  subgraphs}, Discrete Appl. Math. \textbf{156} (2008), no.~7, 1168--1174.

\bibitem{AS}
N.~Alon and J.~Spencer, \emph{The probabilistic method}, third ed.,
  Wiley-Interscience Series in Discrete Mathematics and Optimization, John
  Wiley \& Sons, Inc., Hoboken, NJ, 2008.

\bibitem{BHP}
R.~Baker, G.~Harman, and J.~Pintz, \emph{The difference between consecutive
  primes. {II}}, Proc. London Math. Soc. \textbf{83} (2001), no.~3, 532--562.

\bibitem{NPcomplete2}
A.~Dehghan, M.-R. Sadeghi, and A.~Ahadi, \emph{Algorithmic complexity of proper
  labeling problems}, Theoret. Comput. Sci. \textbf{495} (2013), 25--36.

\bibitem{NPcomplete}
A.~Dudek and D.~Wajc, \emph{On the complexity of vertex-coloring
  edge-weightings}, Discrete Math. Theor. Comput. Sci. \textbf{13} (2011),
  no.~3, 45--50.

\bibitem{FK}
A.~Frieze and M.~Krivelevich, \emph{Packing {H}amilton cycles in random and
  pseudo-random hypergraphs}, Random Structures Algorithms \textbf{41} (2012),
  no.~1, 1--22.

\bibitem{GG}
C.~Godsil and G.~Royle, \emph{Algebraic graph theory}, Graduate Texts in
  Mathematics, vol. 207, Springer-Verlag, New York, 2001.

\bibitem{Handbook}
R.~L. Graham, M.~Gr{\"o}tschel, and L.~Lov{\'a}sz (eds.), \emph{Handbook of
  combinatorics. {V}ol. 1,\ 2}, Elsevier Science B.V., Amsterdam; MIT Press,
  Cambridge, MA, 1995.

\bibitem{JLR}
S.~Janson, T.~{\L}uczak, and A.~Ruci\'nski, \emph{Random graphs},
  Wiley-Interscience Series in Discrete Mathematics and Optimization,
  Wiley-Interscience, New York, 2000.

\bibitem{5enough}
M.~Kalkowski, M.~Karo\'nski, and F.~Pfender, \emph{Vertex-coloring
  edge-weightings: Towards the 1-2-3-conjecture}, J. Comb. Theory, Ser. B
  \textbf{100} (2010), no.~3, 347--349.

\bibitem{KKP}
\bysame, \emph{The 1-2-3 conjecture for hypergraphs}, E-print: arXiv:1308.0611,
  2013.

\bibitem{KLT}
M.~Karo\'{n}ski, T.~{\L}uczak, and A.~Thomason, \emph{Edge weights and vertex
  colours}, J.~Combin. Theory Ser.~B \textbf{91} (2004), 151--157.

\bibitem{Seamone}
B.~Seamone, \emph{The 1-2-3 conjecture and related problems: a survey},
  E-print: arXiv:1211.5122, 2012.

\bibitem{13enough}
T.~Wang and Q.~Yu, \emph{On vertex-coloring 13-edge-weighting}, Front. Math.
  China \textbf{3} (2008), no.~4, 581--587.

\end{thebibliography}

\providecommand{\bysame}{\leavevmode\hbox to3em{\hrulefill}\thinspace}
\providecommand{\MR}{\relax\ifhmode\unskip\space\fi MR }
\providecommand{\MRhref}[2]{%
  \href{http://www.ams.org/mathscinet-getitem?mr=#1}{#2}
}
\providecommand{\href}[2]{#2}

\end{document}